\documentclass[11pt]{amsart}
\usepackage{amssymb}
\usepackage{amsmath}
\numberwithin{equation}{section}

\theoremstyle{plain}

\newtheorem{thm}{Theorem}[section]

\newtheorem{lem}{Lemma}[section]
\newtheorem*{mainthm}{Theorem}
\newtheorem*{mainprop}{Proposition}

\theoremstyle{remark}
\newtheorem{rem}{Remark}[section]

\newcommand{\calI}{\mathcal{I}}

\newcommand{\calL}{\mathcal{L}}

\newcommand{\calO}{\mathcal{O}}

\newcommand{\bC}{\mathbb{C}}

\newcommand{\bH}{\mathbb{H}}

\newcommand{\bR}{\mathbb{R}}

\renewcommand{\Re}{\operatorname{Re}}
\renewcommand{\Im}{\operatorname{Im}}
\newcommand{\Orth}{\operatorname{O}}
\newcommand{\Uni}{\operatorname{U}}
\newcommand{\CO}{\operatorname{CO}}
\newcommand{\CU}{\operatorname{CU}}
\newcommand{\Ric}{\operatorname{Ric}}
\newcommand{\Scal}{\operatorname{Scal}}
\newcommand{\Ein}{\operatorname{Ein}}

\renewcommand{\th}{\theta}
\newcommand{\up}{\Upsilon}

\newcommand{\conj}{\overline}
\newcommand{\ol}{\overline}
\newcommand{\pa}{\partial}
\newcommand{\wt}{\widetilde}
\newcommand{\wh}{\widehat}

\title[Normal form for 
 pseudo-Einstein contact froms]{Normal form for 
 pseudo-Einstein contact forms and intrinsic CR normal coordinates}
\author{Kengo Hirachi}

\address{Graduate School of Mathematical Sciences, The University of Tokyo, 3-8-1 Komaba, Meguro-ku, Tokyo 153-8914 JAPAN}
\email{hirachi@ms.u-tokyo.ac.jp}



\begin{document}
\maketitle
\begin{abstract}
We give a normal form for pseudo-Einstein contact forms  and apply it to construct intrinsic CR normal coordinates parametrized by the structure group of CR geometry. 
The proof is based on the construction of parabolic normal coordinates by Jerison and Lee.
\end{abstract}
\section{Introduction}
Normal coordinates are basic tools in geometric analysis, which give optimal approximations by the flat model and simplify technical computations.
In Riemannian geometry, geodesic normal coordinates $x$ are the canonical choice; they are parametrized by the orthogonal group $\Orth(n)$ and the metric tensor satisfies
\[
g_{ij}=\delta_{ij}+O(|x|^{2}).
\]
If we allow conformal changes of the metric $\wh g=e^{2\up}g$, then we can get  better approximations,
which are called {\em conformal normal coordinates}.  They are parametrized by the structure group of conformal geometry, $\CO(n)\ltimes\bR^{n}$, where $\CO(n)=\Orth(n)\times\bR_{+}$ is the conformal orthogonal group.

There are (at least) two standard choices.  The first one was given by Robin Graham in his study of local conformal invariant \cite{FG}.  He found that, for each point $p$, there is a conformal scale $\up$ such that
the symmetrized covariant derivatives of the Ricci tensor of $\wh g$ vanish to the infinite order:
\[
\Ric_{(a_{1}a_{2},a_{3}\dots a_{k})}(p)=0\quad \text{for  }k\ge2.
\]
Such a scale $\up$ is uniquely determined once the first jets
\[
(e^{\up(p)},\pa_{j}\up(p))\in\bR_{+}\times \bR^{n}
\]
 are specified.  Then the geodesic normal coordinates for $\wh g$ give  conformal normal coordinates.  Together with the choice of an orthonormal frame of the tangent space at $p$, the coordinates are parametrized by $\Orth(n)\times\bR_{+}\times\bR^{n}$.  

Another class of conformal normal coordinates was introduced by Lee and Parker \cite{LP} in the analysis of Yamabe functional.  They normalized the scale $\up$ by the condition:
 \[
  \det \wh g_{ij}(x)=1+O(|x|^{k})\quad \text{for  }k\ge2.
  \]
Again, such a scale is determined by the first jets of $\up$.  

In CR geometry, or the biholomorphic geometry of real hypersurfaces in  a complex manifold,
normal coordinates were first introduce by Moser \cite{CM}. He imposed a condition on the defining function of the surface and fixed holomorphic coordinates up to an action of the structure group, $\CU(n)\ltimes\bH^{n}$, where $\CU(n)=\Uni(n)\times\bR_{+}$ is the conformal unitary group and $\bH^{n}=\bC^{n}\times\bR$ is the Heisenberg group.
Moser's normal coordinates were used in the invariant theory of CR geometry by Fefferman \cite{F}, but they are not easy to handle in the setting of intrinsic pseudo-hermitian geometry.

  Another class of CR normal coordinates was constructed by Jerison and Lee \cite{JL} in the study of CR Yamabe problem, in analogy with the conformal case.  They first fixed a contact form by a curvature condition and then make {\em parabolic normal coordinates} for the scale; see \S3.  The curvature condition is rather complicated: let 
  \begin{align*}
&Q_{\alpha\beta}=(n+2)i A_{\alpha\beta}, \quad
Q_{\alpha\ol\beta}=\Ric_{\alpha\ol\beta},\\
&Q_{0\alpha}=4A_{\alpha\beta,}{}^{\beta}+\frac{2i}{n+1}\Scal_{\alpha},\\
&Q_{00}=\frac{16}{n}\Im A_{\alpha\beta,}{}^{\alpha\beta}+\frac{4}{n(n+1)}\Delta_{b}\Scal,
\end{align*}
where
$\Ric_{\alpha\ol\beta}$, $\Scal$ and $A_{\alpha\beta}$ are respectively the Ricci tensor, the scalar curvature and the torsion tensor of the Tanaka-Webster connection for a contact form $\th$; see \S2.
The indices preceded by a comma denote covariant derivatives and $\Delta_{b}$ is the sublaplacian.
If we use indices $I,J$ which run through $\{0,1\dots, n,\ol 1,\dots,\ol n\}$, then the two tensors can be written in a unified form $Q_{IJ}$; for the components which are note defined above, we set
 $Q_{\ol I\ol J}=\ol{Q_{IJ}}$, $Q_{IJ}=Q_{JI}$.
 Then the normalization is given by the symmetrized covariant derivatives of $Q_{IJ}$:
\begin{equation}\label{Qsym}
Q_{(I_{1}I_{2},I_{3}\dots I_{k})}(p)=0\quad \text{for } k\ge 2.
\end{equation}
The contact forms satisfying this condition are parametrized by
$\bR_{+}\times\bH^{n}$, which is the first jets of the scale.
Then a choice of orthonormal frame of CR tangent bundle $T^{1,0}_{p}$,  parametrized by $\Uni(n)$, determines parabolic normal coordinates.   Hence such coordinates are parametrized by
the structure group $\CU(n)\ltimes\bH^{n}$.

This construction gives intrinsic CR normal coordinates that are useful for asymptotic anlysis of the CR Yamabe functional, but the properties of the normal scale $\th$ is not easy to understand.  Moreover, while the recent progress of CR geometry gives more focus on pseudo-Einstein contact forms in connection with $Q$ and $Q$-prime curvatures \cite{CG, CY, H, HMM}, the normalization \eqref{Qsym} is not compatible with the Einstein equations (see \S2)
\begin{equation}\label{Einstein-eq}
\Ric_{\alpha\ol\beta}=\frac{1}{n}\Scal\cdot h_{\alpha\ol\beta},
\quad  \Scal_{\alpha}=inA_{\alpha\beta,}{}^\beta.
\end{equation}
In contrast with the conformal case, these Einstein equations always have solutions (at least as $\infty$-jets at a point; see Remark \ref{remark}), and it is natural to choose contact forms within this class.  Another merit to work within the class of pseudo-Einstein contact forms is that the scaling functions are restricted to CR pluriharmonic functions.  It already gives a strong normalization on the scale and very much simplifies the computation.

To state our result, let
\[A_{\alpha 0}=A_{0\alpha}:=\frac{-i}{n+1}A_{\alpha\beta,}{}^{\beta},
\quad 
A_{00}:=\frac{-1}{n(n+1)}A_{\alpha\beta,}{}^{\alpha\beta}.
\]
If we use indices $I,J\in\{0,1,\dots,n\}$, we obtain a symmetric two tensor $A_{IJ}$.
We say that $\th$ is {\em formally pseudo-Einstein} at $p$ if the Einstein equations \eqref{Einstein-eq} hold to the infinite order at  $p$.

\begin{mainthm}
Let $M$ be a strictly pseudoconvex CR manifold of dimension $2n+1$.
For each point $p\in M$, there is a formally pseudo-Einstein contact form $\th$ at $p$
such that $\Scal(p)=0$ and 
\[
A_{(I_{1}I_{2},I_{3}\dots I_{k})}(p)=0\quad \text{for }  k\ge2.
\]
Here each $I_{j}$ runs thought $\{0,1,\dots, n\}$.  Moreover, all jets of such $\th$ at $p$ is uniquely determined if one fixes the first jet, which is parametrized by $\bR_{+}\times\bH^{n}$.
\end{mainthm}

Once a contact form is fixed, parabolic normal coordinates are determined by a choice of orthonormal frame of $T_{p}^{1,0}M$; see \S3.  Hence the coordinates are parametrized by the structure group $\CU(n)\ltimes\bH^{n}$.

From the normalization given in the theorem, we can easily observe the vanishing of several derivatives of the curvature and torsion tensors. 

\begin{mainprop}
For a pseudo-Einstein contact form $\th$ for which
 $\Scal$ and $A_{IJ}$ vanish at $p$, 
the following tensors also vanish at $p$.
\begin{equation}
\begin{aligned}\label{R-Ric-normal}
&
\Scal_{\alpha},\quad \Scal_{0},\quad \Scal_{\alpha}{}^{\alpha},  \quad \Scal^{\alpha}{}_{\alpha}, 
\\
&\Ric_{\alpha\conj\beta},\quad \Ric_{\alpha\conj\beta,}{}^{\conj\beta},
\quad \Ric_{\alpha\conj\beta,}{}^{\alpha\conj\beta},\\
&A_{\alpha\beta},\quad A_{\alpha\beta,}{}^{\beta},\quad A_{\alpha\beta,}{}^{\alpha\beta}.
\end{aligned}
\end{equation}
\end{mainprop}

The vanishing of the tensors in \eqref{R-Ric-normal} and
\[
A_{\alpha\beta,\gamma}(p)=A_{(\alpha\beta,\gamma)}(p)=0
\]
 are used in \cite[Theorem 4.1]{JL} to estimate the CR Yamabe functional.   In \cite{JL}, the vanishing of these tensors are derived from \eqref{Qsym} for $k\le 4$ with some computations.  Our construction also simplifies this part of their  argument.

This paper is organized as follows.
In \S2, we quickly review the definition and fundamental facts on the Tanaka-Webster connection and pseudo-Einstein contact forms by following \cite{H, L}.
In  \S3, we recall the parabolic normal coordinates of \cite{JL} and modify them to $\bC^{n+1}$-valued coordinates, which give an approximate CR embedding.  With these coordinates, we give an algorithm for approximating polynomials in the coordinates by CR holomorphic functions; this will be used in the inductive construction of CR pluriharmonic normal scale.
We prove the theorem in \S4 by giving an inductive construction of the jets of pseudo-Einstein contact form.
 We follow \cite[\S3]{JL}  and emphasize on the key steps which need modifications.
The proof of the proposition, which may be obvious from the Einstein equations,  is given at the end.

\medskip

\noindent
{\em Notes.} In his Master's thesis \cite{K}, Satoshi Katsumi claimed that there is a pseudo-Einstein contact form for which the holomorphic derivatives of $\Scal$ and $A_{\alpha\beta}$ vanish at a point $p$:
\[
\Scal_{\alpha_{1}\dots\alpha_{k}}(p)=0,\quad 
A_{\alpha_{1}\alpha_{2},\alpha_{3}\dots\alpha_{k}}(p)=0\quad \text{for all }k.
\]
His argumnet is based on Moser's normal form;  the ambiguity of the normalized contact forms is yet to be studied.
\medskip

\noindent
{\em Notations.}  We adopt the following index conventions:

The lower case Greek indices $\alpha,\beta,\dots$ run though $\{1,2,\dots,n\}$.

The upper case Latin indices $I,J,\dots$ run though $\{0,1,2,\dots,n\}$. 
 
 For a list of indices, we use calligraphic fonts
$\calI=I_{1}\dots I_{k}$ and its length is denoted by $|\calI|=k$.  
The weight $\|\calI\|$ of a list of indices  is defined by
\[
\|I_{1}\dots I_{k}\|=\|I_{1}\|+\cdots+\|I_{k}\|,
\quad 
\|\alpha\|=1 \text{ and }\|0\|=2.
\]
We will also use the index notation of Einstein and Penrose.  Repeated indices are summed:
 \[
f_{\alpha}\th^{\alpha}=\sum_{\alpha=1}^{n} f_{\alpha}\th^{\alpha},
\quad
a_{I}z^{I}=\sum_{I=0}^{n} a_{I}z^{I}.
\]
For the (anti-) symmetrizations of indices, we use   $(\ \ )$ and $[\ \ ]$, e.g.
\[
f_{(\alpha\beta)}=\frac{1}{2}(f_{\alpha\beta}+f_{\beta\alpha}),\quad
f_{[\alpha\ol\beta]}=\frac{1}{2}(f_{\alpha\ol\beta}-f_{\ol\beta\alpha}).
\]
We use overlined indices to denote the conjugate of tensors, e.g.
\[Z_{\ol\alpha}=\ol{Z_{\alpha}}, \quad\th^{\ol\alpha}=\ol{\th^{\alpha}}, \quad\omega_{\ol\alpha}{}^{\ol\beta}=\ol{\omega_{\alpha}{}^{\beta}},
\quad A_{\ol\alpha\ol\beta}=\ol{A_{\alpha\beta}}.
\]

\section{The Tanaka-Webster connection and the \\ pseudo-Einstein condition}
A {\em CR manifold} is a real $(2n+1)$-dimensional manifold $M$ with a distinguished $n$-dimensional integrable
subbundle $T^{1,0}\subset \bC TM$ such that
$T^{1,0}\cap \ol{T^{1,0}}=\{0\}$. For example, if $M$ is real hypersurface in a complex manifold $X$, then $T^{1,0}=T^{1,0}X\cap \bC TM$ gives a structure of CR manifold. In this case, we say $M$ is {\em embeddable.}  
Let $H=\Re T^{1,0}\subset TM$ and take a real one form $\th$  such that $\ker\th=H$.  Then we can define the {\em Levi form} $L_{\th}(X,Y)=-id\th(X,\ol Y)$
for $X,Y\in T^{1,0}$.  If $L_{\th}$ is positive definite, we say $M$ is {\em strictly pseudoconvex}.
In such a case, $\th\wedge(d\th)^{n}\ne0$ and call $\th$ a (positive) {\em contact form.}
Note that any (positive) contact form is given by a scaling $\wh\th=e^{2\up}\th$, $\up\in C^{\infty}(M)$.

In the following, we always assume that $M$ is strictly pseudoconvex and $\th$ is a positive contact form on it.  The {\em Reeb vector field} $T$ for $\th$ is defined by the condition
\[
\th(T)=1,\quad  d\th(T,\cdot)=0.
\] 
We take a local frame $W_{\alpha}$ of $T^{1,0}$.
Then $W_{\alpha},W_{\ol\alpha}, T$ form a local frame of $\bC TM$.
Let $\th^{\alpha},\th^{\ol\alpha},\th$ be its dual coframe, called {\em admissible coframe},
for which 
\[
d\th=i h_{\alpha\ol\beta}\th^{\alpha}\wedge\th^{\ol\beta},
\]
for a positive definite hermitian matrix $h_{\alpha\ol\beta}$.
We will use $h_{\alpha\ol\beta}$ and its inverse $h^{\alpha\ol\beta}$ to lower and raise indices.

A choice of $\th$ determines a linear connection $\nabla$ on $\bC TM$, called the {\em Tanaka-Webster (TW) connection}.
In the frame, $W_{\alpha}, T$, we have $\nabla T=0$ and $\nabla W_{\alpha}=\omega_{\alpha}{}^{\beta}\otimes W_{\beta}$.  The connection from $\omega_{\alpha}{}^{\beta}$  is determined by the following  equations:
\begin{align*}
d\th^{\beta}=\th^{\alpha}\wedge\omega_{\alpha}{}^{\beta}+A^{\beta}{}_{\ol\alpha}\th\wedge\th^{\ol\alpha},\quad \omega_{\alpha\ol\beta}+\omega_{\ol\beta\alpha}=dh_{\alpha\ol\beta}.
\end{align*}
A part of the torsion of $\nabla$ is given by
$A_{\alpha\beta}$, which  is called the {\em TW torsion}. It is shown that 
$
A_{\alpha\beta}=A_{\beta\alpha}$.
The {\em TW curvature} $R_{\alpha\ol\beta\gamma\ol\sigma}$  is defined by the
$\th^{\gamma}\wedge\th^{\ol\sigma}$ part of the curvature form:
\[
d\omega_{\alpha}{}^{\beta}-\omega_{\alpha}{}^{\gamma}\wedge\omega_{\gamma}{}^{\beta}
=R_{\alpha}{}^{\beta}{}_{\gamma\ol\sigma}\th^{\gamma}\wedge\th^{\ol\sigma}\mod \th,\th^{\alpha}\wedge\th^{\beta}, 
\th^{\ol\alpha}\wedge\th^{\ol\beta}.
\]
Other components can be expressed in terms of the torsion.  We then define the {\em Ricci tensor} and the {\em scalar curvature} by
\[
\Ric_{\alpha\ol\beta}=R_{\alpha\ol\beta\gamma}{}^{\gamma},
\quad \Scal=\Ric_{\alpha}{}^{\alpha}.
\]
We will denote the covariant derivatives of a tensor by indices preceded by a comma, e.g.
$A_{\alpha\beta,\gamma}$.  Then a part of the Bianchi identities can be written as
$A_{\alpha\beta,\gamma}=A_{(\alpha\beta,\gamma)}$.  
For a scalar function, we will omit the comma.
For the covariant derivative in $T$, we use the index $0$.  So, a commutative relation of covariant derivative for a function $u$ can be written as
\begin{equation}
u_{\alpha\ol\beta}-u_{\ol\beta\alpha}=ih_{\alpha\ol\beta}u_{0}.
\end{equation}

Under the scaling $\wh\th=e^{2\up}\th$, we choose 
$\wh\th^{\alpha}=\th^{\alpha}+2i\up^{\alpha}\th$ as an admissible coframe.  Then we have
\begin{align*}
\wh A_{\alpha\beta}&=A_{\alpha\beta}+2i \up_{\alpha\beta}-4i \up_{\alpha}\up_{\beta},
\\
\wh\Scal&=e^{-2\up}\big(\Scal-2(n+1)(\Delta_{b}\up+2n\up_{\alpha}\up^{\alpha})\big),
\end{align*}
where $\Delta_{b}\up=\up_{\alpha}{}^{\alpha}+\up^{\alpha}{}_{\alpha}$.

A function $f$ is {\em CR holomorphic} if $f_{\ol\alpha}=0$.  A real valued function $u$ is {\em CR pluriharmonic} if $u$ can be locally written as the real part of a CR holomorphic function.  
We are particularly interested in the scaling by a CR pluriharmonic function $\up=\Re f$.  For a CR holomorphic $f$, we have
$f_{\ol\alpha}=0$ and $
\Delta_{b} f=in f_{0}$; hence the transformation laws above give
\begin{align}
\wh A_{\alpha\beta}&=A_{\alpha\beta}+i f_{\alpha\beta}-i f_{\alpha}f_{\beta},
\label{transAf}\\
\wh\Scal&=e^{-2\up}\big(\Scal-4n(n+1)(-\Im f_{0}+ f_{\alpha}f^{\alpha})\big).
\label{transSf}
\end{align}

Let us define {\em Einstein tensors} by
\begin{align*}
\Ein_{\alpha\conj\beta}&=\Ric_{\alpha\conj\beta}-\frac1n h_{\alpha\conj\beta}\Scal,
\\
\Ein_{\alpha}&= \Scal_{\alpha}-inA_{\alpha\beta,}{}^\beta.
\end{align*}
We say $\th$ is {\em pseudo-Einstein} if 
\begin{equation}\label{pEeq}
\Ein_{\alpha\conj\beta} =0
\quad\text{and}\quad
\Ein_{\alpha}=0.
\end{equation}
When $n\ge 2$, the first equation is the original definition of the pseudo-Einstein condition in \cite{L}.  Since
\begin{align*}
n\Ein_{\alpha\conj\beta,}{}^{\conj\beta}&=(n-1)\Ein_{\alpha},
\end{align*}
the second equation follows from the first.  
When $n=1$, the first equation is trivial and we only have $\Ein_{1}=0$. 
See \cite{CG, HMM} for more geometric aspects of the system \eqref{pEeq}.
We will use the following basic facts \cite{H, L}:
\begin{itemize}
\item[(1)]
If $M$ is embeddable, then a pseudo-Einstein contact form exists locally.
\item[(2)]
If $\th$ is pseudo-Einstein, then $\wh\th=e^{2\up}\th$ is pseudo-Einstein if and only if $\up$ is CR pluriharmonic.
\end{itemize}

\section{Parabolic normal coordinates} 
 Fix a point $p\in M$ and a contact form $\th$ on $M$.
 Using the splitting $TM=H\oplus \bR T$, we write a tangent vector at $p$
as
\[
W+c T,\quad W\in H_{p}, c\in\bR.
\] 
Then we consider a curve $\gamma$ satisfying the ordinary differential equation
\begin{equation}\label{parabolic-geodesic}
\nabla_{\dot\gamma}\dot\gamma=2c T, \quad \gamma(0)=p,\ \dot\gamma(0)=W.
\end{equation}
For  $W+c T$ near $0\in T_p M$, the solution $\gamma=\gamma_{W,c}$ exist on $[0,1]$ and we may define a map
\[
\Psi\colon T_{p}M\to M,\quad \Psi(W+cT)=\gamma_{W,c}(1),
\]
which is shown to be diffeomorphic near $0$.  If we fix an orthonormal frame $W_{\alpha}$ of $T^{1,0}_{p}M$,
we can give coordinates of $T_{p}M$ by $\bH^{n}=\bC^{n}\times\bR$,
\[
\bH^{n}\ni(z^{\alpha},t)\mapsto z^{\alpha}W_{\alpha}+z^{\ol\alpha}W_{\ol\alpha}+t T\in T_{p}M.
\]
Composing with $\Psi$, we obtain a local diffeomorphism
$
\bH^{n}\to M
$.  Its local inverse $M\to\bH^{n}$ defined near $p$ is called  {\em parabolic normal coordinates}.  By taking the parallel transport of $W_{\alpha}$ along each $\gamma$, we may define a local frame, called a {\em special frame}, which is also denoted by $W_{\alpha}$. 

On $\bH^{n}$, there is an action of $s\in\bR_{+}$ given by $\delta_{s}(z^{\alpha},t)=(s z^{\alpha},s^{2}t)$, called {\em  dilations}.  Its generator is
\[
X=z^{\alpha}\pa_{\alpha}+z^{\ol\alpha}\pa_{\ol\alpha}+2t\pa_{t},
\]
where $\pa_{\alpha}=\pa/\pa z^{\alpha}$.
Note that any tensor $\varphi$ defined near $p$ can be decomposed into homogenous parts with respect to the dilations:
\[
\varphi\sim\sum_{m}\varphi_{(m)}, \quad \calL_{X}\varphi_{(m)}=m\varphi_{(m)},
\]
where $\calL_{X}$ denotes the Lie derivative. Here the sum is understand in a formal sense and do not consider its convergence.  If $\varphi=\varphi_{(m)}$, we say that $\varphi$ has {\em weight} $m$. 

 We write $\varphi=\calO_{m}$ if
$\varphi_{(j)}=0$ for all $j<m$.  
On a CR manifold, we can also define the notion of $\calO_{m}$ at $p\in M$ via parabolic normal coordinates.  While the definition depends on the choice of contact form, it is easy to check that $\calO_{m}$ depends only on the contact bundle $H$.  In fact, for a function, $f=\calO_{m}$ if and only if
\[
X_{1}\cdots X_{l}f(p)=0 \ \text{ for any }l<m\text{ and }  X_{j}\in\Gamma(H).
\]
As usual, we write $f=\calO_{\infty}$ if $f=\calO_{m}$ for any $m$.

On $\bH^{n}$, the standard contact form is given by
\[
\th=dt-\frac{i}{2}(z_{\alpha}dz^{\alpha}+z^{\alpha}dz_{\alpha}),
\]
where $z_{\alpha}=\ol{z^{\alpha}}$, the index is lowered by $\delta_{\alpha\ol\beta}$.
Then $d\th=idz^{\alpha}\wedge dz_{\alpha}$ and the Reeb vector field is $T=\pa_{t}$.
The bundle $T^{1,0}$  is defined by the frame
\[
Z_\alpha
=\pa_\alpha+\frac{i}{2}z_\alpha\pa_t,
\]
which is orthonormal for $d\th$.  We also set $Z_{0}=-i T$ and $z^{0}=-|z|^{2}/2+it$, where $|z|^{2}=z^{\alpha}z_{\alpha}$.  (The factor ``$-i$'' in $Z_{0}$ makes computations simpler as $Z_{0}z^{0}=1$.)
Then $Z_{\ol\alpha}z^{0}=0$ and we call $z^{I}$ the {\em complex coordinates} of $\bH^{n}$,
which give an embedding of $\bH^{n}$ into $\{2\Re z^{0}+|z|^{2}=0\}\subset\bC^{n+1}$.

If we use the notation $
Z_{I_{1}\dots I_{k}}=Z_{I_{1}}\cdots Z_{I_{k}}$, we have
\[
Z_{[\alpha\beta]}=0,\quad 2Z_{[\alpha\ol\beta]}=\delta_{\alpha\ol\beta}Z_{0},\quad 
Z_{[\alpha0]}=0. 
\]
It is clear that $z^{I}$ (and $dz^{I}$) has weight $\|I\|$ and $Z_{I}$ has weight $-\|I\|$.
A monomial
\[
z^{\calI}=z^{I_{1}}\cdots z^{I_{k}}
\]
has weight $\|\calI\|$. We now show that a CR holomorphic function has Taylor series in powers of $z^{I}$.

\begin{lem}
Let $f$ be a CR holomorphic function defined near $0\in\bH^{n}$. Then, for any
$m$, the weight $m$ part of $f$ is given by
\[
f_{(m)}=\sum_{\|\calI\|=m}\frac{1}{|\calI|!}z^{\calI}Z_{\calI}f(0).
\]
\end{lem}

\begin{proof}
Identifying $\bH^{n}$ with the boundary of Siegel domain 
\[
\Omega=\{(z^{\alpha},z^{0})\in\bC^{n+1}:2\Re z^{0}+|z|^{2}<0\},
\]
 we take a function $\wt f(z^{\alpha},z^{0})$ defined near $0\in\bC^{n+1}$ which is holomorphic in $\Omega$ such that
$\wt f(z^{\alpha},-|z|^{2}/2+it)=f(z^{\alpha},t)$. Writing the Taylor series of $\wt f$ at $0\in\bC^{n+1}$ in the form
\[
\wt f(z^{\alpha},z^{0})\sim \sum_{\calI}\frac{1}{|\calI|!}z^{\calI}\pa_{\calI}\wt f(0),
\]
we will show $
Z_{\calI}f(0)=\pa_{\calI}\wt f(0).
$
Since $Z_{\alpha}$ can be extended to $\pa_{\alpha}+z_{\alpha}\pa_{0}$ on $\bC^{n+1}$ and $z_{\alpha}$ commutes with $\pa_{I}$, we may replace $Z_{\alpha}$ by $\pa_{\alpha}$ when evaluated at $0$.  
For a holomorphic function, we have $\pa_{0}\wt f=-i\pa_{t}\wt f$ and hence $Z_{0}$ can be replaced by $\pa_{0}$.
\end{proof}

Now we consider parabolic normal coordinates $(z^{\alpha}, t)$ centered at $p\in M$
and set 
\[
z^{0}=-|z|^{2}/2+it.
\]  We call $z^{I}=(z^{\alpha},z^{0})$ the {\em complex parabolic coordinates}. With the coordinates, we can also define $Z_{\alpha}$ and $Z_{0}=-i\pa_{t}$ on $M$. The differences between $Z_{I}$ and the special frame $W_{I}$ (we set $W_{0}=-iT$) have  estimate (\cite[Proposition 2.5]{JL})
\[
W_{\alpha}=Z_{\alpha}+\calO_{1}, \quad
W_{0}=Z_{0}+\calO_{0}.
\]
It follows that the coordinates  $z^{I}$ satisfy
$W_{\ol\alpha}\,z^{I}=\calO_{\|I\|+1}.$
Hence, for a monomial, the Leibniz rule gives
\begin{equation}\label{monomial-estimate}
W_{\ol\alpha}\,z^{\calI}=\calO_{\|\calI\|+1}.
\end{equation}
While a polynomial in $z^{I}$ may not be CR holomorphic, we can approximate it by a CR holomorphic function as follows.

\begin{lem}\label{correction-lem}
Let $f$ be a function satisfying $W_{\ol\alpha}f=\calO_{m}$. Then there exists a function $g= \calO_{m+1}$ such that
$
W_{\ol\alpha}(f+g)=\calO_\infty$.
Moreover, if $M$ is embeddable,  there is a CR holomorphic function $f_{m}$ such that
$f_{m}=f+\calO_{m+1}$.
\end{lem}
\begin{proof}
We set
\[
g_{1}=\sum_{l=1}^{m+1}
\frac{(-1)^{l}}{l!}z^{\ol\alpha_{1}}\cdots z^{\ol\alpha_{l}}
W_{\ol\alpha_{1}\cdots\ol\alpha_{l}}f.
\]
The summands are $\calO_{m+1}$ as $z^{\ol\alpha_{1}}\cdots z^{\ol\alpha_{l}}W_{\ol\alpha_{1}\dots \ol\alpha_{l-1}}=\calO_{1}$ and $W_{\ol\alpha_{l}}f=\calO_{m}$.
Using $W_{\ol\alpha}z^{\ol\beta}=\delta_{\ol\alpha}^{\ol\beta}+\calO_{2}$ and the integrability condition
$
W_{[\ol\alpha\ol\beta]}=\varphi_{\ol\alpha\ol\beta}^{\ol\gamma}W_{\ol\gamma},
$
we get
\[
W_{\ol\alpha_{0}}(f+ g_{1})=\frac{(-1)^{m+1}}{(m+1)!} z^{\ol\alpha_{1}\cdots \ol\alpha_{m+1}}
W_{\ol\alpha_{0}\ol\alpha_{1}\dots\ol\alpha_{m+1}}f+\calO_{m+1},
\]
which is also $\calO_{m+1}$ as the derivatives of $f$ are in $\calO_{0}$.  Repeating this procedure, we can find $g_{k}=\calO_{m+k}$ such that
$f_{k}=f+g_{1}+\cdots+g_{k}$ satisfies $W_{\ol\alpha}f_{k}\in\calO_{m+k}$.  Then $\sum_{k}g_{k}$
defines a formal power series at $p$ and Borel's lemma gives the required $g$.

Now we consider the case $M$ is realized, near $p=0$, as a surface 
\[
\rho=2u+|z|^{2}+F(z,\ol z, t)=0,
\]
where $(z^{\alpha},w)=(z^{\alpha},u+it)\in\bC^{n+1}$ and $F$ vanishes to the 3rd order at $0$.  Moreover, by linear coordinates change, we can also assume that $W_{\alpha}z^{\beta}=\delta_{\alpha}^{\beta}$ at $0$. By restriction, we regard $(z^{\alpha},w)$ as functions on $M$.  Since the contact form is $\calO_{2}$ and
\[
\th=-\frac{i}{2}(\pa-\ol\pa)\rho=dt+\calO_{2},
\] we have $t=\calO_{2}$.  On the other hand, from the equation $2u=-|z|^{2}-F$ on $M$, we  have $u=\calO_{2}$ and thus $w=\calO_{2}$.

With these holomorphic coordinates  $(z^{\alpha},w)$ (which are, in general, different from the parabolic coordinates),  
we consider the Taylor expansion of $f$ up to $\calO_{m+1}$:
\[
f_{m}=\sum_{ p+q+2l\le m}
C^{l}_{\alpha_{1}\dots\alpha_{p}\ol\beta_{1}\dots\ol\beta_{q}}z^{\alpha_{1}\dots\alpha_{p}}z^{\ol\beta_{1}\dots\ol\beta_{q}}w^{l}.
\]
Since $W_{\ol\alpha}\calO_{m+1}\subset\calO_{m}$,
the assumption $W_{\ol\alpha}f=\calO_{m}$ froces $W_{\ol\alpha}f_{m}=\calO_{m}$. 
But, by $W_{\ol\alpha}z^{\ol\beta}=\delta_{\alpha}^{\beta}+\calO_{1}$, we see that 
 $f_{m}$ can not contain $z^{\ol\beta}$.  Thus $f_{m}$ is a holomorphic polynomial. 
\end{proof}

If we apply this lemma to complex parabolic coordinates $z^{I}$, we obtain $\wt z^{I}$ such that
\[
\wt z^{\alpha}=z^{\alpha}+\calO_{3}, \quad \wt z^{0}=z^{0}+\calO_{4}\quad\text{and}\quad
W_{\ol\alpha}\wt z^{I}=\calO_{\infty}. 
\]
The algorithm for constructing $\wt z^{I}$ is given in the lemma and we can write down the Taylor expansion of $\wt z^{I}$ in the coordinates $(z^{\alpha},t)$ in terms of the jets of $R_{\alpha\ol\beta\gamma\ol\sigma}$ and $A_{\alpha\beta}$; see \cite[Proposition 2.5]{JL}.

\begin{rem}\label{remark}
The coordinates $\wt z^{I}$ give a $C^{\infty}$ embedding of $M$ into $\bC^{n+1}$ as a surface
\[
\wt M=\{(z,u+iv)\in\bC^{n+1}:
2u+|z|^{2}+F(z,\ol z, v)=0\}
\]
and the $\infty$-jets of CR structure at $p$ is isomorphic to that for  $\wt M$ at $0$.
So, the computation depending only on the jets of CR structure at a point $p\in M$ can be done equivalently for 
$\wt M$ at $0$.  For such calculations, we lose no generality by assuming the CR embeddability.  
For example, we obtain  a formally pseudo-Einstein contact form at $p$ from the pull-back of the one on $\wt M$ and the the statement (2) in the end of \S2 also holds for {\em formally CR pluriharmonic functions}, which are real parts of functions $f$ such that $W_{\ol\alpha}f=\calO_{\infty}$.
\end{rem}

\section{Proof of the theorem}
The proof is done by the induction on the weight $m$. For each step, we prove the following

\begin{thm}
Let $\th_{0}$ be pseudo-Einstein contact from defined near $p\in M$.
For any $m\ge2$, there is a CR pluriharmonic function $\up$ such that
the  scalar curvature and the torsion for $\th=e^{2\up}\th_{0}$ satisfies  
\begin{equation}\label{torsion-normalization}
\Scal(p)=0\quad\text{and}\quad
 A_{(I_{1}I_{2},I_{3}\cdots I_{k})}(p)=0,\quad \|I_{1}\dots I_{k}\|\le m.
 \end{equation}
Moreover,  $\up(p)$ and $\up_{I}(p)$ can be chosen arbitrary and, once they are fixed, 
$\up$ modulo $\calO_{m+1}$ is uniquely determined.
\end{thm}

Here we have assumed the existence of a pseudo-Einstein contact form in order to avoid the term ``formally''.
As we explained in Remark \ref{remark},  we will lose no generality by this change.
 
The first step is the case $m=2$.   Setting $2\Re \up=f+\ol f$ for a CR holomorphic function
with $f(0)\in\bR$,
we see that  $\up(p)$, $\up_{\alpha}(p)$ and $\up_0(p)$ correspond to $f(p)$, $f_{\alpha}(p)$ and
$\Re f_{0}(p)$ up to factors of $2$.

\begin{lem}
For given $(f(p), f_{\alpha}(p), \Re f_{0}(p))\in\bR\times\bC^{n}\times\bR$, one may take a CR holomorphic $f$ so that $\th=e^{2\Re f}\th_{0}$ satisfies 
\begin{equation}\label{first-step}
\Scal(p)=0,\quad  A_{\alpha\beta}(p)=0.
\end{equation}
Moreover, such $f$ is unique modulo $\calO_{3}$.
\end{lem}

\begin{proof} 
We use complex parabolic normal coordinates $z^I$ for $\th_{0}$ centered at $p$.
Let us fix $g=c+a_{\alpha}z^{\alpha}$ and set $\th=e^{2\Re g}\th_{0}$.
Then, under the scaling $\wh\th=e^{2\Re u}\th$ for a polynomial of weight $2$
 \[
 u=a_{\alpha\beta}z^\alpha z^\beta+a_{0} z^{0},
 \]
we obtain from  the transformation law \eqref{transAf} and \eqref{transSf} that
\begin{align*}
\wh\Scal&=\Scal-4n(n+1)\Re a_{0}+\calO_{2},
\\
\wh A_{\alpha\beta}&=A_{\alpha\beta}+i a _{\alpha\beta}+\calO_{2}.
\end{align*}
Thus we can fix $\Re a_{0}$ and $a_{\alpha\beta}$ so that $\wh\Scal(p)=0$ and $\wh A_{\alpha\beta}(p)=0$ hold.
Note that $\Im a_{0}=-\Re u_{0}(p)$ can be taken arbitrary.

Let $f=g+u$.
Then from the estimate \eqref{monomial-estimate}, we have $W_{\ol\alpha}f=\calO_{2}$. Thus we can use Lemma \ref{correction-lem} to modify $f$ to a CR holomorphic function by adding $\calO_{3}$-term, which does not change the normalization \eqref{first-step}.
\end{proof}

We next compute the effect of terms of $\calO_m$ in the scaling.

\begin{lem}\label{variationAf}
Let  $m\ge2$. For a CR holomorphic function $f=\calO_{m}$, set $\wh\th=e^{2\Re f}\th$. Then the following approximate transformation laws, computed in a fixed spacial from $W_{\alpha}$ for $\th$, hold:
\begin{equation}\label{A-variation}
\wh A_{IJ}-A_{IJ}=i ^{\|IJ\|-1}Z_{IJ}f+\calO_{m+2-\|IJ\|}.
\end{equation}
\end{lem}
\begin{proof}

Recall from \cite[Lemma 3.6]{JL} that, for a scaling by $\up\in\calO_{m}$, we have
\begin{align*}
\wh A_{\alpha\beta}&-A_{\alpha\beta}=2i\up_{\alpha\beta}+\calO_{m},\\
\wh A_{\alpha\beta,}{}^{\beta}&-A_{\alpha\beta,}{}^{\beta}=2i\up_{\alpha\beta}{}^{\beta}+\calO_{m-1},\\
\wh A_{\alpha\beta,}{}^{\alpha\beta}&-A_{\alpha\beta,}{}^{\alpha\beta}=
2i\up_{\alpha\beta}{}^{\alpha\beta}+\calO_{m-2}.
\end{align*}
In our setting, $2\up=f+\ol f$ and hence $2\up_{\alpha}=f_{\alpha}$,
$2\up_{\alpha\beta}=f_{\alpha\beta}$.
Using the commutation relations of covariant derivatives (see e.g. \cite[Lemma 2.3]{L}) and $f^{\alpha}=0$, we obtain
\begin{align*}
f_{\alpha\beta}{}^{\beta}&=i(n+1)f_{\alpha0}+\calO_{m-1},
\\
f_{\alpha\beta}{}^{\beta\alpha}&=i(n+1)f_{\alpha0}{}^{\alpha}+\calO_{m-2}
=-n(n+1)f_{00}+\calO_{m-2}.
\end{align*}
Comparing with the normalization in the definition of $A_{IJ}$ and substituting $T=iZ_{0}$, 
we get \eqref{A-variation}.
\end{proof}

By induction, suppose that we have a pseudo-Einstein contact form $\th$ satisfying \eqref{torsion-normalization}
with $m$ replaced by $m-1$.  We first choose $f=\calO_{m}$ as a polynomial of weight $m$
so that \eqref{torsion-normalization} holds.

For $\calI=I_{1}\dots I_{k}$ of weight $m$, we apply $\nabla_{I_{k}\dots I_{3}}$ to the both  sides of \eqref{A-variation}.
Then,  using $T=iZ_{0}$, we get
\[
\wh A_{\calI}=A_{\calI}+i^{m-|\calI|+1}Z_{\calI}f+\calO_{2}.
\]
(Note that $m-|\calI|$ is the number of $0$ in the list $\calI$.)
Hence $\wh A_{(\calI)}=\calO_{1}$ for all $\calI$ with weight $m$ if and only if
\[
f=-\sum_{\|\calI\|=m}\frac{i^{m-|\calI|+1}}{|\calI|!}A_{\calI}z^{\calI}+\calO_{m+1}.
\]
We define $f_{m}$ by this formula. Then \eqref{monomial-estimate} gives
$W_{\ol\alpha}f_{m}=\calO_{m+2}$.  Thus we can make $\calO_{m+1}$-correction and get a CR holomorphic $f$ satisfying \eqref{torsion-normalization}.  This completes the proof of the theorem.

\begin{rem}
In the inductive step above, we have changed the parabolic normal coordinates according to the modification of the contact forms and hence we  need to estimate the errors caused by that. We have omitted this part as it is exactly same as \cite[\S3]{JL}. 
\end{rem}

We will conclude the paper with the proof of the proposition stated in the introduction.
For a pseudo-Einstein contact form, 
$
\Scal_{\alpha},
\Scal_{\alpha}{}^{\alpha},
  \Scal_{0}
  $ are respectively constant multiples of
  $A_{\alpha0},  A_{00}, \Re A_{00}$, which vanish at $p$.
  Thus we are done for the case $n=1$. 
  When $n\ge2$, the Einstein equation $n\Ric_{\alpha\conj\beta}=\Scal h_{\alpha\conj\beta}$
gives
 \[
n \Ric_{\alpha\conj\beta,}{}^{\conj\beta}=\Scal_{\alpha},\quad
 n\Ric_{\alpha\conj\beta,}{}^{\alpha\conj\beta}=\Scal_{\alpha}{}^{\alpha}.
 \]
We have already seen that the left-hand sides vanish at $p$ and  the proof is completed.

\section*{Acknowledgements}
This research was supported by KAKENHI 20H00116.

\end{document}